\documentclass[11pt,reqno]{amsart}

\usepackage{amsmath,amssymb,mathrsfs,amsthm}
\usepackage{graphicx,cite,cases}
\usepackage{xcolor}
\usepackage{bm}
\newtheorem{theorem}{Theorem}[section]

\newtheorem{lemma}[theorem]{Lemma}
\newtheorem{proposition}[theorem]{Proposition}
\newtheorem{remark}[theorem]{Remark}

\newtheorem{definition}{Definition}[section]

\numberwithin{equation}{section}

\begin{document}

\title[nonradiating sources of Maxwell's equations]{Nonradiating sources of Maxwell's equations}

\author{Peijun Li}
\address{LSEC, ICMSEC, Academy of Mathematics and Systems Science, Chinese Academy of Sciences, Beijing 100190, China, and School of Mathematical Sciences, University of Chinese Academy of Sciences, Beijing 100049, China.}
\email{lipeijun@lsec.cc.ac.cn}

\author{Jue Wang}
\address{School of Mathematics, Hangzhou Normal University, Hangzhou 311121, China}
\email{wangjue@hznu.edu.cn}

\thanks{The second author is supported by NSFC (No.12371420) and the Natural Science Foundation of Zhejiang Province (No. LY23A010004).}

\subjclass[2010]{78A40, 78A46}

\keywords{Maxwell's equations, nonradiating sources, Green's function, inverse source problem, nonuniqueness.}

\begin{abstract}
This paper presents a thorough investigation into nonradiating sources of Maxwell's equations. Various characterizations are developed to clarify the properties of nonradiating sources, considering their varying degrees of regularity. Furthermore, the characterizations are examined on far-field patterns and near-field data of the electric field, along with the null spaces of integral operators. The study includes the explicit construction of several illustrative examples to demonstrate the presence of nonradiating sources in different null spaces.
\end{abstract}

\maketitle

\section{Introduction}\label{S:in}

In the field of electromagnetism, the dynamics of electric and magnetic fields are governed by Maxwell's equations, which are propelled by electric currents. A nonradiating source refers to a configuration of electric currents that refrains from emitting electromagnetic radiation. The characterization of nonradiating sources constitutes a classical problem in electromagnetic theory, with far-reaching implications for diverse applications such as antenna design, stealth technology, and electromagnetic compatibility.

It is recognized that surface measurements of electromagnetic fields cannot uniquely determine volume sources, mainly because nonradiating sources are inherently present \cite{DS1982, KLR2005}. A nonradiating source must satisfy specific conditions to ensure that the associated fields do not propagate energy to infinity \cite{A2004, EA2004}. The existence of nonradiating sources carries significant implications for the inverse source problem in electromagnetism. This problem involves determining the distribution of electric currents based on measurements of electromagnetic fields at a specified distance, often at the boundary of a domain \cite{GF2003}. Understanding the limitations imposed by nonradiating sources is crucial for correctly interpreting the results of inverse source problems and for designing measurement systems that can effectively deal with these challenges. Inverse source problems are encountered in various scientific and engineering disciplines, including areas such as medical imaging, environmental monitoring, and remote sensing.

Significant research efforts have been dedicated to investigating the inverse source problem related to a variety of wave equations, including acoustic, electromagnetic, elastic, and plate waves. From mathematical and computational perspectives, the uniqueness problem poses a challenge in solving inverse source problems. Because of the presence of nonradiating sources, the inverse source problem at a specific frequency is inherently ill-posed. Consequently, it becomes impossible to deduce the wave source uniquely based solely on a single far-field pattern or boundary data at a fixed frequency \cite{EV2009, MD1999, EAR2000}. The concept of nonradiating sources was initially introduced in \cite{BC1977}, where the problem of nonuniqueness for the inverse source problem in acoustics and electromagnetics was investigated. Nonradiating sources for both the acoustic and elastic wave equations were studied in \cite{B2018, B2019}. To address the challenge posed by nonuniqueness, the work presented in \cite{BLLT2011} offers a comprehensive review of the subject matter related to inverse scattering problems. The focus is on the utilization of multi-frequency boundary measurements as a strategic approach. In addition to uniqueness, the study of stability has received significant attention in inverse source problems, as evidenced by the works in \cite{BLT2010, BLZ2020, CIL2016, LY2017}. Recently, research has undertaken the exploration of the more intricate inverse source problems associated with stochastic wave equations, as exemplified by \cite{BCL2016, LLW2022}, wherein the sources are modeled as random fields rather than deterministic functions.

The inverse source problems associated with Maxwell's equations, particularly the magnetoencephalograph (MEG) problem, have been extensively investigated with the objective of determining the current density from surface measurements of the electric or magnetic field \cite{NOHTA2007}. In the study conducted by \cite{HR1998}, certain uniqueness results were established, along with the derivation of explicit formulas aimed at identifying the location and moment of the dipole source. Similarly, \cite{RLR2002} examined uniqueness issues and explored numerical solutions for reconstructing a current distribution from the magnetic field. In \cite{AYV2004} and \cite{GAF2005}, discussions on the non-uniqueness of the inverse MEG problem were provided, emphasizing its generic limitation of the inverse source problem. Utilizing the variational approach outlined in \cite{AM2006}, it was shown that the inverse source problem typically lacks a unique solution. In \cite{V2012}, a simple uniqueness result was presented for identifying the location and shape of an electromagnetic source function using multiple frequency information. \cite{ABF2002} introduced a computational method to determine the current dipole from boundary measurements of the fields, based on a low-frequency asymptotic analysis of Maxwell's equations. In the work by \cite{BN2013}, a numerical approach was formulated to reconstruct multiple point dipoles using boundary measurements of both electric and magnetic fields. The Fourier method was suggested for solving the multi-frequency electromagnetic inverse source problem in \cite{WMGL2018}. In \cite{LY2005}, a conditional stability was investigated for the inverse source problem of the time-dependent Maxwell's equations in anisotropic media. In a recent study by \cite{LW2021}, the inverse random source problem for Maxwell's equations was investigated, illustrating that the diagonal entries of the strength matrix for the random source can be uniquely determined by utilizing the amplitude of the electric field.

Motivated by \cite{BC1977}, the work presented in \cite{LW2023} investigated characterizations of nonradiating sources for the two- and three-dimensional biharmonic wave equations. The objective of this study is to offer diverse characterizations of nonradiating sources associated with Maxwell's equations. Owing to the distinct natures inherent in the biharmonic wave equation and Maxwell's equations, especially due to differences in their corresponding Green's functions, the analytical approach and characterizations necessarily diverge. We define a collection of criteria that both determine and are required for a source to exhibit nonradiating behavior. This is accomplished through the utilization of the integral representation of the electric field and the expansion of the dyadic Green's function for Maxwell's equations by using spherical harmonic functions. We provide descriptions of nonradiating sources by taking account of sources exhibiting different levels of regularity. Additionally, we discuss descriptions that incorporate far-field patterns and near-field data of the electric field, along with the null spaces of integral operators. Lastly, we construct several illustrative examples explicitly to show the presence of nonradiating sources. These instances highlight that the null spaces are nontrivial and emphasize that the inverse source problem is inherently ill-posed, lacking a unique solution at a given frequency.

The structure of the paper is as follows: In Section \ref{S:Max}, we introduce Maxwell's equations and the problem formulation. Section \ref{S:Eqc} provides a detailed exploration of the characterizations of nonradiating sources in the context of Maxwell's equations. In Section \ref{S:ex}, various examples of nonradiating sources are presented. Finally, concluding remarks are discussed in Section \ref{S:co}.

\section{Maxwell's equations}\label{S:Max}

Consider the time-harmonic Maxwell's equations in three dimensions
\begin{align}\label{MaEq}
\nabla\times\boldsymbol{H}=-\mathrm{i}\omega\epsilon\boldsymbol{E}
+\boldsymbol J_0, \quad \nabla\times\boldsymbol{E}
=\mathrm{i}\omega\mu\boldsymbol{H},
\end{align}
where $\omega> 0$ denotes the angular frequency, $\epsilon$ and $\mu$ are positive constants representing the electric permittivity and magnetic permeability, respectively, $\boldsymbol{E}$ is the electric field, $\boldsymbol{H}$ denotes the magnetic field, and $\boldsymbol J_0$ represents the electric current, assumed to be compactly supported within the domain $D$ enclosed in the ball $B_R=\{x\in\mathbb R^3: |x|<R\}$. It is noteworthy that $\boldsymbol J_0$ is regarded as a general vector field and is not assumed to be divergence free, i.e., $\nabla\cdot\boldsymbol J_0\neq 0$. As usual, the Silver--M\"{u}ller radiation condition is enforced to guarantee the well-posed nature of the scattering problem
\begin{equation}\label{SMrc}
\lim_{|x|\to\infty}((\nabla\times\boldsymbol{E})\times x-{\rm i}\kappa |x|\boldsymbol{E})=0,
\end{equation}
where $\kappa=\omega\sqrt{\epsilon\mu}$ is the wavenumber.

It is known that a solution $\boldsymbol E$ to \eqref{MaEq}--\eqref{SMrc} exhibits the asymptotic behavior given by
\begin{align}\label{Easb}
\boldsymbol E(x)= \frac{e^{{\rm i}\kappa|x|}}{|x|}\bigg[ {\boldsymbol E}_{\infty} ({\hat x})+O\bigg(\frac{1}{|x|}\bigg)\bigg],\quad |x|\to \infty,
\end{align}
where ${\hat x}={x}/|x|$ is the unit observation direction vector, and ${\boldsymbol E}_{\infty}$ is referred to as the far-field pattern of ${\boldsymbol E}$.

By removing $\boldsymbol H$ from \eqref{MaEq}, we derive the Maxwell system corresponding to $\boldsymbol E$:
\begin{equation}\label{Eqal-1}
 \nabla\times(\nabla\times\boldsymbol{E})-\kappa^2 \boldsymbol E=\boldsymbol J,
\end{equation}
where $\boldsymbol J={\rm i}\omega\mu\boldsymbol J_0$ has the same compact support $D$ contained in $B_R$. It follows from the identity $\nabla\times(\nabla\times\boldsymbol{E})=-\Delta\boldsymbol E+\nabla(\nabla\cdot\boldsymbol E)$ that \eqref{Eqal-1} can be equivalently written as 
\begin{align}\label{Eqal-2}
\nabla(\nabla\cdot\boldsymbol{E})-(\Delta\boldsymbol{E} +\kappa^2\boldsymbol{E})
=\boldsymbol{J}. 
\end{align}
Using the identity $\nabla\cdot[\nabla\times(\nabla\times\boldsymbol{E})]=0$ and \eqref{Eqal-1}, we have
\begin{align}\label{Eqal-2-1}
-\kappa^2\nabla(\nabla\cdot{\boldsymbol E})=\nabla(\nabla\cdot{\boldsymbol J}).
\end{align}
Combining \eqref{Eqal-2} and \eqref{Eqal-2-1} leads to 
\begin{equation}\label{Eqal-2-3}
 \Delta\boldsymbol{E} +\kappa^2\boldsymbol{E}=-(\boldsymbol J+ \kappa^{-2}\nabla(\nabla\cdot{\boldsymbol J})).
\end{equation}
Noting $\nabla\times(\nabla(\nabla\cdot{\boldsymbol E}))=0$, we obtain from \eqref{Eqal-1} and \eqref{Eqal-2-1} that 
\begin{align} \label{Eqal-2-2}
\nabla\times\nabla\times(\boldsymbol{E}+ \kappa^{-2}\nabla(\nabla\cdot{\boldsymbol E}))
-\kappa^2 (\boldsymbol{E}+ \kappa^{-2}\nabla(\nabla\cdot{\boldsymbol E}))=\boldsymbol J+ \kappa^{-2}\nabla(\nabla\cdot{\boldsymbol J}).
\end{align}
Hence, the function $\boldsymbol J+ \kappa^{-2}\nabla(\nabla\cdot{\boldsymbol J})$ can be considered as the source for both the Helmholtz equation \eqref{Eqal-2-3} and the Maxwell equation \eqref{Eqal-2-2}.

Let $g$ represent the Green's function of the three-dimensional Helmholtz equation, defined by 
\[
 \Delta g(x,y)+\kappa^2 g(x,y)=-\delta(x-y),
\]
where $\delta$ is the Dirac delta function. Specifically, the expression for $g(x,y)$ is given as
\[
 g(x,y)=\frac{e^{{\rm i}\kappa|x-y|}}{4\pi|x-y|}. 
\]
Denote by $\boldsymbol{G}$ the dyadic Green's tensor associated with Maxwell's equations, satisfying
\begin{equation}\label{Eqal-3}
\nabla\times\nabla\times\boldsymbol{G}(x, y)-\kappa^{2}\boldsymbol{G}(x, y)=\delta(x-y)\boldsymbol{I}, 
\end{equation}
where $\boldsymbol{I}$ is the $3\times 3$ identity matrix. Explicitly, the dyadic Green's function $\boldsymbol{G}$ can be expressed as
\begin{align}\label{DGM}
 \boldsymbol{G}(x, y)=\big[\boldsymbol{I}+\kappa^{-2}\nabla\nabla\big]g(x,y)=(G_{l,s}(x,y))_{l,s=1,2,3}. 
\end{align}

A straightforward calculation yields 
\begin{align*}
\frac{\partial g(x, y)}{\partial x_{l}}=-\frac{\partial g(x, y)}{\partial y_{l}}=\frac{(x_{l}-y_{l})}{|x- y|}\bigg(\mathrm{i}\kappa -\frac{1}{|x- y|}\bigg)g(x, y),\quad\quad\quad\quad l=1,2,3,
\end{align*}
and
\begin{align*}
\frac{\partial^2 g(x, y)}{\partial x_{l}\partial x_{s}}&=\frac{\partial^2 g(x, y)}{\partial y_{l}\partial y_{s}}
=\bigg[\frac{(\mathrm{i}\kappa|x- y|-1)}{|x- y|^2}\delta_{ls}-\frac{\kappa^2(x_{l}-y_{l})(x_{s}-y_{s})}{|x- y|^2}
\nonumber\\
&\quad -\frac{3\mathrm{i}\kappa(x_{l}-y_{l})(x_{s}-y_{s})}{|x- y|^3}
+\frac{3(x_{l}-y_{l})(x_{s}-y_{s})}{|x-y|^4} \bigg]g(x, y),\quad l,s=1,2,3.
\end{align*}
Hence, for $l,s=1,2,3$, we obtain 
\begin{align*}
&G_{ls}(x, y)=\delta_{ls}g(x, y)+\frac{1}{\kappa^2}\frac{\partial^2 g(x, y)}{\partial y_{l}\partial y_{s}}\nonumber\\
&=\bigg[\bigg(\delta_{ls}-\frac{(x_{l}-y_{l})(x_{s}-y_{s})}{|x- y|^2}\bigg)
+\frac{(\mathrm{i}\kappa|x- y|-1)}{\kappa^2|x- y|^2}\bigg(\delta_{ls}-\frac{3(x_{l}-y_{l})(x_{s}-y_{s})}{|x- y|^2}\bigg) \bigg]g(x, y),
\end{align*}
where $\delta_{ls}$ denotes the Kronecker delta.

Utilizing \eqref{Eqal-3} and noting that the support of $\boldsymbol J$ is confined to $B_R$, we obtain that the solution to \eqref{MaEq}--\eqref{SMrc} can be expressed as
\begin{align}\label{EIN}
\boldsymbol E (x)=\int_{B_R}\boldsymbol G(x, y) \boldsymbol{J}(y){\rm d}y,\quad x\in\mathbb R^3.
\end{align}
If $\nabla(\nabla\cdot{\boldsymbol  J})$ is well-defined, by using \eqref{Eqal-2-3} and \eqref{Eqal-2-2}, we can deduce 
\begin{align}\label{EIN-gr1}
\boldsymbol E (x) =\int_{B_R} g(x, y) \big(\boldsymbol J (y)+ \kappa^{-2}\nabla(\nabla\cdot{\boldsymbol J}(y))\big){\rm d}y,\quad x\in\mathbb R^3,
\end{align}
and
\begin{align}\label{EIN-gr}
\boldsymbol E (x)+\kappa^{-2}\nabla(\nabla\cdot{\boldsymbol E} (x))
=\int_{B_R}\boldsymbol G(x, y) \big(\boldsymbol J (y)+ \kappa^{-2}\nabla(\nabla\cdot{\boldsymbol J}(y))\big){\rm d}y,\quad x\in\mathbb R^3.
\end{align}

Let us provide a formal definition of a nonradiating source in the context of Maxwell's equations.

\begin{definition}\label{def-E}
Consider a source denoted as $\boldsymbol{J}$, with a compact support denoted as $D$, entirely confined in the ball $B_R$. A source is classified as nonradiating if the solution $\boldsymbol{E}$ to \eqref{EIN} is identically zero outside of the ball $B_R$.
\end{definition}

\begin{lemma}\label{Lemm-Hrc} 
The solution $\boldsymbol E$ to \eqref{EIN-gr1} satisfies the Silver--M\"{u}ller radiation condition \eqref{SMrc}.
\end{lemma}

\begin{proof}
From \eqref{DGM} and \eqref{EIN-gr1}, we have from a straightforward calculation that 
\begin{align}\label{EIN-gr1-1}
\boldsymbol E (x) &=\int_{B_R} g(x, y) \big(\boldsymbol J (y)+ \kappa^{-2}\nabla(\nabla\cdot{\boldsymbol J}(y))\big){\rm d}y\nonumber\\
&=\int_{B_R} \big[\big(\boldsymbol{I}+\kappa^{-2}\nabla\nabla\big)g(x, y)\big] \big(\boldsymbol J (y)+ \kappa^{-2}\nabla(\nabla\cdot{\boldsymbol J}(y))\big){\rm d}y\nonumber\\
&\quad
-\int_{B_R}  [\kappa^{-2}\nabla\nabla g(x, y)] \big(\boldsymbol J (y)+ \kappa^{-2}\nabla(\nabla\cdot{\boldsymbol J}(y))\big){\rm d}y\nonumber\\
&=\int_{B_R} \boldsymbol G(x, y)  \big(\boldsymbol J (y)+ \kappa^{-2}\nabla(\nabla\cdot{\boldsymbol J}(y))\big){\rm d}y \notag\\
&\quad -\kappa^{-2}\nabla\nabla\cdot\int_{B_R}   g(x, y) \big(\boldsymbol J (y)+ \kappa^{-2}\nabla(\nabla\cdot{\boldsymbol J}(y))\big){\rm d}y\nonumber\\
&=\int_{B_R} \boldsymbol G(x, y)  \big(\boldsymbol J (y)+ \kappa^{-2}\nabla(\nabla\cdot{\boldsymbol J}(y))\big){\rm d}y
-\kappa^{-2}\nabla(\nabla\cdot{\boldsymbol E} (x)),\quad x\in\mathbb R^3.
\end{align}
Noting that $\kappa^{-2}\nabla\cdot{\boldsymbol E} (x)=\nabla\cdot{\boldsymbol J} (x)$ and $\boldsymbol J$ has compact support $D$ contained in the ball $B_R$, we arrive at 
\begin{align}\label{EIN-gr1-2}
\nabla\cdot{\boldsymbol E} (x)=0,\quad |x|>R.
\end{align}
Combining \eqref{EIN-gr1-1} and \eqref{EIN-gr1-2}, we obtain
\begin{align*}
\boldsymbol E (x) =\int_{B_R} \boldsymbol G(x, y)  \big(\boldsymbol J (y)+ \kappa^{-2}\nabla(\nabla\cdot{\boldsymbol J}(y))\big){\rm d}y,\quad |x|>R,
\end{align*}
which implies that \eqref{EIN-gr1} satisfies the Silver--M\"{u}ller radiation condition \eqref{SMrc}.
\end{proof}

\section{Characterizations of nonradiating sources}\label{S:Eqc}

In this section, we present characterizations of nonradiating sources of Maxwell's equations, taking into account sources with varying degrees of regularity. Furthermore, we elaborate on characterizations that involve far-field patterns and near-field data of the electric field, as well as the null spaces of integral operators.

\subsection{Characterizations with different regularity}

Define the space of square-integrable functions and divergence 
\begin{align*}
H(\mathrm{div};B_R)=\big\{\boldsymbol{\Psi}\in (L^2(B_R))^3: \nabla\cdot\boldsymbol{\Psi}\in L^2(B_R)\big\},
\end{align*}
and its subspace
\begin{align*}
H_0(\mathrm{div};B_R)=
\big\{\boldsymbol{\Psi}\in H(\mathrm{div};B_R): (\boldsymbol{\Psi}\cdot\boldsymbol{n})|_{\partial B_R}=0\big\},
\end{align*}
where $\boldsymbol n$ stands for the unit normal vector at $\partial B_R$ pointing outward from $B_R$. Denote the space
\begin{align*} 
H(\mathrm{div}; \mathrm{grad}; B_R)=\big\{\boldsymbol{\Psi}\in (L^2(B_R))^3 : \nabla(\nabla\cdot\boldsymbol{\Psi})\in (L^2(B_R))^3\big\}.
\end{align*}

By the addition theorem (cf. \cite[Theorem 2.11]{DR-2013}), the Green's function $g$ for the three-dimensional Helmholtz equation admits the expansion
\begin{align}\label{Gex}
g(x, y)=\frac{\exp{(\mathrm{i}\kappa|x-y|)}}{4\pi|x -y|}
=\mathrm{i}\kappa\sum_{n=0}^{\infty}\sum_{m=-n}^{n}h_{n}^{(1)}(\kappa|x|)Y_{n}^{m}({\hat x})j_{n}(\kappa|y|)\overline{Y_{n}^{m}({\hat y})},\quad |x|>|y|,
\end{align}
where ${\hat x}={x}/|x|$ and ${\hat y}={y}/|y|$, the spherical harmonics $\{Y_{n}^{m}: m=-n, \dots, n, n=0, 1, \dots\}$ form a complete orthonormal system in the space of square integrable functions on the unit sphere, $j_n$ and $y_n$ represent the spherical Bessel function and spherical Neumann function of order $n$, respectively, and $h_{n}^{(1)}:=j_n+{\rm i}y_n$ is recognized as the spherical Hankel function of the first kind of order $n$.

It follows from \eqref{DGM}--\eqref{EIN} and \eqref{Gex} that we have
\begin{align}\label{EINE}
\boldsymbol E (x)&=\int_{B_R}\boldsymbol  G(x, y) \boldsymbol{J}(y){\rm d}y=\int_{B_R}\Big(\big(\boldsymbol{I}+\kappa^{-2}\nabla_{y}\nabla_{y}\big)g(x, y)\Big)\boldsymbol{J}(y){\rm d}y\nonumber\\
&=\mathrm{i}\kappa \sum_{n=0}^{\infty}\sum_{m=-n}^{n} h_{n}^{(1)}(\kappa|x|)Y_{n}^{m}({\hat x})\beta_{n}^{m},\quad\quad\quad |x|>|y|,
\end{align}
where
\begin{align}\label{EINE-0}
\beta_{n}^{m}
&=\int_{B_R}\Big(\big(\boldsymbol{I}+\kappa^{-2}
\nabla_{y}\nabla_{y} \big)j_{n}(\kappa|y|)\overline{Y_{n}^{m}({\hat y})}\Big)\boldsymbol{J}(y){\rm d} y
:=\alpha_n^m +\gamma_n^m
\end{align}
with
\begin{align}
\label{alpha}\alpha_n^m &=\int_{B_R} j_{n}(\kappa|y|)\overline{Y_{n}^{m}({\hat y})}\boldsymbol{J}(y){\rm d}y,\\
\label{gamma}\gamma_n^m &=\frac{1}{\kappa^2}\int_{B_R}\big(\nabla_{y}\nabla_{y}j_{n}(\kappa|y|)\overline{Y_{n}^{m}({\hat y})} \big)\boldsymbol{J}(y){\rm d}y.
\end{align}

Next, we discuss characterizations of nonradiating sources, considering variations in their degrees of regularity.

(1) For $\boldsymbol{J}\in (L^2(B_R))^3$, it follows from \eqref{EINE}--\eqref{EINE-0} that $\boldsymbol{J}$ is nonradiating if and only if $\beta_n^m=\alpha_n^m+\gamma_n^m=0$ for all $n=0, 1, \dots, m=-n, \dots, n$.

(2) For $\boldsymbol{J}\in H_0(\mathrm{div};B_R)$, we have from $\eqref{gamma}$ and the divergence theorem that 
\begin{align}\label{EINE-11}
\gamma_{n}^{m} &=-\frac{1}{\kappa^2}\int_{B_R}\big[\big(\nabla_y j_{n}(\kappa|y|)\overline{Y_{n}^{m}({\hat y})} \big)(\nabla\cdot{\boldsymbol  J}(y))\big]{\rm d}y\nonumber\\
&\quad+\frac{1}{\kappa^2}\int_{\partial B_R}\big[\big(\nabla_y j_{n}(\kappa|y|)\overline{Y_{n}^{m}({\hat y})} \big)(\boldsymbol  n(y)\cdot{\boldsymbol  J}(y))\big]{\rm d}y\nonumber\\
&=-\frac{1}{\kappa^2}\int_{B_R}\big[\big(\nabla_y j_{n}(\kappa|y|)\overline{Y_{n}^{m}({\hat y})} \big)(\nabla\cdot{\boldsymbol  J}(y))\big]{\rm d}y:=\zeta_n^{m}, 
\end{align}
which implies that $\boldsymbol{J}$ is nonradiating if and only if $\beta_n^m=\alpha_n^m+\zeta_n^m=0$ for all $n=0, 1, \dots, m=-n, \dots, n$.

(3) For $\boldsymbol{J}\in H(\mathrm{div}; \mathrm{grad}; B_R)$,  we obtain from \eqref{EIN-gr1} and \eqref{Gex} that 
\begin{align*}
\boldsymbol  E (x)
&=\int_{B_R} g(x, y) \big(\boldsymbol  J (y)+ \kappa^{-2}\nabla(\nabla\cdot{\boldsymbol  J}(y))\big){\rm d}y\nonumber\\
&=\mathrm{i}\kappa\sum_{n=0}^{\infty}\sum_{m=-n}^{n}h_{n}^{(1)}(\kappa|x|)Y_{n}^{m}({\hat x})\int_{B_R}\big(j_{n}(\kappa|y|)\overline{Y_{n}^{m}({\hat y})}\big) \big(\boldsymbol  J (y)+\kappa^{-2}\nabla(\nabla\cdot{\boldsymbol  J}(y))\big){\rm d}y\nonumber\\
&=\mathrm{i}\kappa\sum_{n=0}^{\infty}\sum_{m=-n}^{n}h_{n}^{(1)}(\kappa|x|)Y_{n}^{m}({\hat x})({\alpha}_n^{m}
+\eta_n^{m}),\quad |x|>R,
\end{align*}
where
\begin{align*}
\eta_{n}^{m}=\frac{1}{\kappa^2}\int_{B_R}j_{n}(\kappa|y|)\overline{Y_{n}^{m}({\hat y})}\big(\nabla(\nabla\cdot{\boldsymbol  J}(y))\big){\rm d}y.
\end{align*}
Hence, the source $\boldsymbol{J}$ is nonradiating if and only if $\beta_n^m=\alpha_n^m+\eta_n^m=0$ for all $n=0, 1, \dots$ and $m=-n, \dots, n$.

Building upon the preceding discussion, we deduce the following characterizations for nonradiating sources of Maxwell's equations, accounting for varying degrees of regularity.

\begin{theorem}\label{The-123}
Consider $\boldsymbol{J}$ as a source for the Maxwell equation \eqref{Eqal-2}. The following results hold:

\begin{enumerate}
\item[(i)] The source $\boldsymbol{J}\in (L^2(B_R))^3$ is nonradiating if and only if $\beta_n^{m}=\alpha_n^{m}
+\gamma_n^{m}=0$ for all $n = 0,1,\dots, m = -n, \dots, n$.

\item[(ii)] The source $\boldsymbol{J}\in H_0(\mathrm{div};B_R)$ is nonradiating if and only if $\beta_n^{m}=\alpha_n^{m}+\zeta_n^{m}=0$ for all $n = 0,1,\dots, m = -n, \dots, n$.

\item[(iii)] The source $\boldsymbol{J}\in H(\mathrm{div}; \mathrm{grad}; B_R)$ is nonradiating if and only if $\alpha_n^{m}+\eta_n^{m}=0$ for all $n = 0,1,\dots, m = -n, \dots, n$.
\end{enumerate}
\end{theorem}

It is worth mentioning that if the source $\boldsymbol J$ is divergence free, i.e., $\nabla\cdot{\boldsymbol  J}=0$, then we deduce from \eqref{EINE-11} that $\zeta_n^m=\eta_n^m=0$ and $\beta_{n}^{m}=\alpha_n^{m}$. 

\begin{proposition}\label{The-alph} 
Let $\boldsymbol{J}$ be a source for the Maxwell equation \eqref{Eqal-2}. If $\nabla\cdot{\boldsymbol  J}=0$, then
the source $\boldsymbol{J}\in (L^2(B_R))^3$ is nonradiating if and only if $\beta_n^{m}=\alpha_n^{m}=0$ for all $n = 0,1,\dots, m = -n, \dots, n$.
\end{proposition}

\subsection{Far-field patterns}

This section is to investigate the characterization of nonradiating sources by utilizing the far-field patterns of the electric field. 

For $y\in B_R$ and $|x|\to\infty$, taking into account the asymptotic expansions
\begin{align}\label{aex-y}
|x-y| = |x|- {\hat x} \cdot y+O\left(\frac{1}{|{x}|}\right)
\end{align}
and
\begin{align}\label{aexy-1}
\frac{1}{|x- y|}=\frac{1}{|x|}\bigg[1+O\left(\frac{1}{|{x}|}\right)\bigg],
\end{align}
we derive from \eqref{aex-y}--\eqref{aexy-1} that
\begin{align}\label{gase}
g(x, y)=\frac{\exp{(\mathrm{i}\kappa|x-y|)}}{4\pi|x-y|}=\frac{e^{\mathrm{i}\kappa|x|}}{4\pi|x|}\bigg[e^{-\mathrm{i}\kappa {\hat x} \cdot y}+O\bigg(\frac{1}{| {x}| }\bigg)\bigg]
\end{align}
and
\begin{align*}
G_{ls}(x, y)&=\bigg[\delta_{ls}-\frac{(x_{l}-y_{l})(x_{s}-y_{s})}{|x-y|^2}+\frac{(\mathrm{i}\kappa|x-y|-1)}{\kappa^2|x-y|^2}\bigg(\delta_{ls}-\frac{3(x_{l}-y_{l})(x_{s}-y_{s})}{|x- y|^2}\bigg) \bigg]\times\nonumber\\
&\quad\quad
\frac{e^{{\rm i}\kappa|x|}}{4\pi|x|}\bigg[e^{-\mathrm{i}\kappa {\hat x} \cdot y}+O\bigg(\frac{1}{| {x}|}\bigg)\bigg]\nonumber\\
&=\bigg[\delta_{ls}-\frac{(x_{l}-y_{l})(x_{s}-y_{s})}{|x-y|^2}+O\bigg(\frac{1}{|{x}|}\bigg)\bigg(\delta_{ls}-\frac{3(x_{l}-y_{l})(x_{s}-y_{s})}{|x- y|^2}\bigg) \bigg]\times\nonumber\\
&\quad\quad
\frac{e^{{\rm i}\kappa|x|}}{4\pi|x|}\bigg[e^{-\mathrm{i}\kappa {\hat x} \cdot y}+O\bigg(\frac{1}{| {x}|}\bigg)\bigg]\nonumber\\
&=\frac{e^{{\rm i}\kappa|x|}}{4\pi|x|}\bigg[\tilde{G}_{l,s}(x, y)e^{-\mathrm{i}\kappa {\hat x} \cdot y}+O\bigg(\frac{1}{| {x}|}\bigg)\bigg],\quad l,s=1,2,3.
\end{align*}
Consequently, we have
\begin{align}\label{GMaxaeM}
\boldsymbol{G}(x, y)
=\frac{e^{{\rm i}\kappa|x|}}{4\pi|x|}\bigg[\tilde{\boldsymbol G}(x, y)
e^{-{\rm i}\kappa {\hat x} \cdot y} +O\bigg(\frac{1}{| {x}| }\bigg)
\bigg],\quad |x|\to \infty,\ y\in B_R,
\end{align}
where
\begin{align*}
\tilde{\boldsymbol G}(x, y)
=\begin{pmatrix}
1-\frac{(x_{1}-y_1)^2}{|x- y|^2} & -\frac{(x_{1}-y_1)(x_{2}-y_2)}{|x-y|^2} & -\frac{(x_{1}-y_1)(x_{3}-y_3)}{| x- y|^2}
\\[5pt]
-\frac{(x_{1}-y_1)(x_{2}-y_2)}{|x- y|^2}& 1-\frac{(x_{2}-y_2)^2}{|x- y|^2} & -\frac{(x_{2}-y_2)(x_{3}-y_3)}{| x-y|^2}
\\[5pt]
-\frac{(x_{1}-y_1)(x_{3}-y_3)}{|x- y|^2} & -\frac{(x_{2}-y_2)(x_{3}-y_3)}{|x- y|^2} & 1-\frac{(x_{3}-y_3)^2}{|x-y|^2}
\end{pmatrix}.
\end{align*}

For $\boldsymbol{J}\in (L^2(B_R))^3$, it follows from \eqref{EIN} and \eqref{GMaxaeM} that 
\begin{align}\label{EIMax}
\boldsymbol  E (x)
=\frac{e^{{\rm i}\kappa|x|}}{4\pi|x|}\bigg[\int_{B_R}e^{-{\rm i}\kappa {\hat x} \cdot y} (\tilde{\boldsymbol G}(x, y)
\boldsymbol{J}(y)){\rm d}y+O\bigg(\frac{1}{|x|}\bigg)\bigg],\quad|x|\to \infty.
\end{align}
For $l,s=1,2,3$, straightforward calculations yield
\begin{align}\label{aexy-2}
&\frac{(x_{l}-y_{l})(x_{s}-y_{s})}{|x-y|^2}=\frac{(x_l- y_l)(x_s-y_s)}{|x|^2\bigg[1-2\bigg(\frac{{\hat x} \cdot y}{|x|}-\frac{|y|^2}{2|x|^2}\bigg)\bigg]}\nonumber\\
&=\bigg({\hat x}_l-\frac{y_l}{|x|}\bigg)\bigg({\hat x}_s-\frac{y_s}{|x|}\bigg)\bigg[1-2\bigg({\hat x} \cdot\frac{y}{|x|}-\frac{1}{2}\bigg|\frac{y}{|x|}\bigg|^2\bigg)\bigg]^{-1}\nonumber\\
&={\hat x}_l {\hat x}_s +O\bigg(\frac{1}{|x|}\bigg),\quad|x|\to \infty. 
\end{align}
Combining \eqref{Easb} and \eqref{EIMax}--\eqref{aexy-2} leads to 
\begin{align}\label{Emax-inf}
{\boldsymbol  E}_{\infty} ({\hat x})
=\int_{B_R} e^{-{\rm i}\kappa {\hat x} \cdot y} \big(\hat{\boldsymbol G}(\hat x)\boldsymbol{J}(y)\big){\rm d}y
=\hat{\boldsymbol G}(\hat x)\int_{B_R} e^{-{\rm i}\kappa {\hat x} \cdot y} \boldsymbol{J}(y){\rm d}y,
\end{align}
where $\hat x=x/|x|=(\hat x_1, \hat x_2, \hat x_3)$ is a unit vector, and
\begin{align}\label{hGtG}
\hat{\boldsymbol G}(\hat x)
=\begin{pmatrix}
1-\frac{x_{1}^2}{|x|^2} & -\frac{x_{1}x_{2}}{|x|^2} & -\frac{x_{1}x_{3}}{|x|^2}\\
-\frac{x_{1}x_{2}}{|x|^2}& 1-\frac{x_{2}^2}{|x|^2} & -\frac{x_{2}x_{3}}{|x|^2} \\
-\frac{x_{1}x_{3}}{|x|^2} & -\frac{x_{2}x_{3}}{|x|^2} & 1-\frac{x_{3}^2}{|x|^2}
\end{pmatrix}
=\begin{pmatrix}
1- \hat x_1^2& -\hat x_1 \hat x_2 & -\hat x_1 \hat x_3\\
-\hat x_1 \hat x_2 & 1- \hat x_2^2 & -\hat x_2 \hat x_3 \\
-\hat x_1 \hat x_3 & -\hat x_2\hat x_3 & 1-\hat x_3^2
\end{pmatrix}.
\end{align}

Recall the Jacobi--Anger expansion for the plane wave (cf. \cite[Theorem 2.8 and $(2.46)$]{DR-2013}):
\begin{align*}
e^{{\rm i}\kappa{x}\cdot{\hat d}}=4\pi\sum_{n=0}^{+\infty}
\sum_{m=-n}^{n} {{\rm i}}^{n}j_{n}(\kappa|x|)Y_n^m({\hat x})\overline{Y_n^m({\hat  d})},\quad x\in\mathbb{R}^3,
\end{align*}
where $\hat{d}\in\mathbb R^3$ is a unit propagation direction vector. Then, for $\xi\in\mathbb{R}^3$, we have
\begin{align}\label{plJAmax-1}
e^{-{\rm i}{\xi}\cdot{x}}
=e^{-{\rm i}{|\xi|{\hat{\xi}}}\cdot{x}}
=e^{-{\rm i}{|\xi|{x}\cdot{\hat{\xi}}}}
=4\pi\sum_{n=0}^{+\infty}
\sum_{m=-n}^{n}
(-{{\rm i}})^{n}j_{n}(|{\xi}||x|)\overline{Y_n^m({\hat x})}Y_n^m({\hat\xi}),\quad x\in\mathbb{R}^3,
\end{align}
where $\hat{\xi}=\xi/|\xi|$.  Using \eqref{Emax-inf}--\eqref{plJAmax-1}, we obtain
\begin{align}\label{plJAmax-2}
\boldsymbol {E}_{\infty} ({\hat x})
&=\hat{\boldsymbol G}(\hat x)\int_{B_R} e^{-{\rm i}\kappa {\hat x} \cdot y} \boldsymbol{J}(y){\rm d}y
\nonumber\\
&=\hat{\boldsymbol G}(\hat x)\int_{B_{R}}\left(4\pi\sum_{n=0}^{+\infty}
\sum_{m=-n}^{n} (-{{\rm i}})^{n}j_{n}(\kappa|y|)\overline{Y_n^m({\hat y})}Y_n^m({\hat x})\right)\boldsymbol{J}(y){\rm d}y\nonumber\\
&=4\pi\sum_{n=0}^{+\infty}
\sum_{m=-n}^{n}(-{\rm i})^{n}\big(\hat{\boldsymbol G}(\hat x)\alpha_n^m\big) Y_n^m({\hat x}), 
\end{align}
where $\alpha_n^m$ is defined in \eqref{alpha}.

\begin{lemma}\label{lemm-tp}
For any $|x|>0$, $\hat{\boldsymbol G}(\hat x)\alpha_n^m=0$ if and only if $\alpha_n^{m}=0$ for all $n = 0,1,\dots, m = -n, \dots, n$.
\end{lemma}

\begin{proof}
In the spherical coordinates, we have 
\begin{align*}
 \hat x_1=\sin\theta\cos\varphi,\quad \hat x_2=\sin\theta\sin\varphi,\quad \hat x_3=\cos\theta,
\end{align*}
where $\theta\in[0,\pi], \varphi\in[0,2\pi]$. The matrix $\hat{\boldsymbol G}(\hat x)$ in \eqref{hGtG} can be written in the form $\hat{\boldsymbol G}(\hat x):=\hat{\boldsymbol G}(\theta,\varphi)$. For $(\theta, \varphi)=(0, 0)$ and $(\theta, \varphi)=(\frac{\pi}{4}, \frac{\pi}{4})$, the matrices are defined as 
\begin{align*}
\hat{\boldsymbol G}(0,0)
=\begin{pmatrix}
1& 0 &0\\
0& 1 &0\\
0& 0 &0
\end{pmatrix},\quad
\hat{\boldsymbol G}\bigg(\frac{\pi}{4},\frac{\pi}{4}\bigg)
=\begin{pmatrix}
\frac{3}{4}& -\frac{1}{4} &-\frac{\sqrt{2}}{4}\vspace{0.2cm}\\
-\frac{1}{4}&  \frac{3}{4}&-\frac{\sqrt{2}}{4}\vspace{0.2cm}\\
-\frac{\sqrt{2}}{4}& -\frac{\sqrt{2}}{4} &\frac{1}{2} 
\end{pmatrix}.
\end{align*}
It can be easily verified from $\hat{\boldsymbol G}(0,0)\alpha_n^m=\hat{\boldsymbol G}(\frac{\pi}{4},\frac{\pi}{4})\alpha_n^m=0$ that $\alpha_n^m=0$  for all $n = 0,1,\dots, m = -n, \dots, n$.  Conversely, the fact is evident.
\end{proof}

By applying Lemma \ref{lemm-tp}, we derive the following characterization of nonradiating sources based on the far-field patterns.

\begin{theorem}\label{T-FFP-1} 
For all $n = 0,1,\dots, m = -n, \dots, n$, if $\gamma_n^{m}=0$. Then, the source $\boldsymbol{J}\in (L^2(B_R))^3$ is nonradiating if and only if
its far-field pattern ${\boldsymbol E}_{\infty} ({\hat x})$ in \eqref{plJAmax-2}  is equal to zero.
\end{theorem}

In the scenario where $\boldsymbol{J}\in H(\mathrm{div}; \mathrm{grad}; B_R)$, define 
\[
\boldsymbol {\mathcal{J}}(x)=\boldsymbol J (x)+ \kappa^{-2}\nabla(\nabla\cdot{\boldsymbol J} (x)).
\]
Subsequently, we derive an alternative characterization of nonradiating sources for Maxwell's equations.  

\begin{theorem}\label{The-FE}
The source $\boldsymbol{J}\in H(\mathrm{div}; \mathrm{grad}; B_R)$ is nonradiating if and only if 
$\boldsymbol {{\mathcal{\hat J}}}=0$ when $|\xi|=\kappa$, where $\boldsymbol {{\mathcal{\hat J}}}$ represents the Fourier transform of $\boldsymbol{\mathcal J}$. 
\end{theorem}

\begin{proof}

Using \eqref{plJAmax-1}, we obtain
\begin{align*} 
\boldsymbol {{\mathcal{\hat J}}}(\xi)
&=\int_{B_{R}}\boldsymbol {\mathcal{J}}(x) e^{-{\rm i} \xi \cdot x}{\rm d}x
=\int_{B_{R}}\boldsymbol {\mathcal{J}}(x) \left(4\pi\sum_{n=0}^{+\infty}
\sum_{m=-n}^{n}
(-{{\rm i}})^{n}j_{n}(|{\xi}||x|)\overline{Y_n^m({\hat x})}Y_n^m({\hat\xi})\right){\rm d}x
\nonumber\\
&=4\pi\sum_{n=0}^{+\infty}
\sum_{m=-n}^{n}(-{\rm i})^{n}Y_n^m({\hat\xi})\int_{B_{R}} j_{n}(|{\xi}||x|)\overline{Y_n^m({\hat x})}\boldsymbol {\mathcal{J}}(x){\rm d}x\nonumber\\
&=4\pi\sum_{n=0}^{+\infty}
\sum_{m=-n}^{n}(-{\rm i})^{n}Y_n^m({\hat\xi})
\int_{0}^{R}j_{n}(|{\xi}|r)r^2\bigg(\int_{0}^{\pi}\int_{0}^{2\pi}{\boldsymbol {\mathcal{J}}}(r,\theta,\phi) \overline{Y_n^m(\theta,\phi)}\sin\theta{\rm d} \theta{\rm d}\phi\bigg){\rm d}r
\nonumber\\
&=4\pi\sum_{n=0}^{+\infty}\sum_{m=-n}^{n}(-{\rm i})^{n}Y_n^m({\hat\xi})\int_{0}^{R}{\boldsymbol{\mathcal{J}}}_n^{m}(r)j_{n}(|{\xi}|r)r^2{\rm d}r,
\end{align*}
which implies
\begin{align}\label{FJ-1}
\boldsymbol {{\mathcal{\hat J}}}({\xi})
=4\pi\sum_{n=0}^{+\infty}
\sum_{m=-n}^{n}(-{\rm i})^{n} (\alpha_n^{m}
+\eta_n^{m}) Y_n^m({\hat\xi}),\quad |\xi|=\kappa.
\end{align}
The proof is completed by utilizing Theorem \ref{The-123} and \eqref{FJ-1}.
\end{proof}

\begin{theorem}\label{The-pc}
The source $\boldsymbol{J}\in H(\mathrm{div}; \mathrm{grad}; B_R)$ is nonradiating if and only if $\boldsymbol {E}_{\infty} ({\hat x})= 0$.
\end{theorem}

\begin{proof}
From \eqref{EIN-gr1} and  \eqref{gase}, for $\boldsymbol{J}\in H(\mathrm{div}; \mathrm{grad}; B_R)$, we can derive
\begin{align}\label{EIN-inf-1}
\boldsymbol  E (x)
&=\int_{B_R} g(x, y) \big(\boldsymbol  J (y)+ \kappa^{-2}\nabla(\nabla\cdot{\boldsymbol  J}(y))\big){\rm d}y\nonumber\\
&=\frac{e^{\mathrm{i}\kappa|x|}}{4\pi|x|}\int_{B_R}\bigg[e^{-\mathrm{i}\kappa {\hat x} \cdot y}+O\bigg(\frac{1}{| {x}| }\bigg)\bigg] \big(\boldsymbol  J (y)+ \kappa^{-2}\nabla(\nabla\cdot{\boldsymbol  J}(y))\big){\rm d}y\nonumber\\
&=\frac{e^{\mathrm{i}\kappa|x|}}{4\pi|x|}\bigg[\int_{B_R}e^{-\mathrm{i}\kappa {\hat x} \cdot y}\big(\boldsymbol  J (y){\rm d} y+ \kappa^{-2}\nabla(\nabla\cdot{\boldsymbol  J}(y))\big)+O\bigg(\frac{1}{| {x}| }\bigg)\bigg] ,\quad \  |x|\to \infty.
\end{align}
It follows from \eqref{Easb}, \eqref{plJAmax-1}, and \eqref{EIN-inf-1} that
\begin{align}\label{plJA-2}
\boldsymbol {E}_{\infty} ({\hat x})&=\int_{B_R}e^{-{\rm i}\kappa {\hat x} \cdot y}\big(\boldsymbol  J (y)+ \kappa^{-2}\nabla(\nabla\cdot{\boldsymbol  J} (y))\big){\rm d} y\nonumber\\
&=4\pi\sum_{n=0}^{+\infty} \sum_{m=-n}^{n}(-{\rm i})^{n}Y_n^m({\hat x})
\int_{B_{R}} j_{n}(\kappa|y|)\overline{Y_n^m({\hat y})}\big(\boldsymbol  J (y)+ \kappa^{-2}\nabla(\nabla\cdot{\boldsymbol  J} (y))\big){\rm d}y.
\end{align}
Noting
\begin{align}\label{plJA-3}
&\int_{B_{R}} j_{n}(\kappa|y|)\overline{Y_n^m({\hat y})}\big(\boldsymbol  J (y)+\kappa^{-2}\nabla(\nabla\cdot{\boldsymbol  J} (y))\big){\rm d}y=\alpha_n^{m}+\eta_n^{m},
\end{align}
and combining \eqref{plJA-2} and \eqref{plJA-3}, we obtain
\begin{align}\label{plJA-4}
\boldsymbol {E}_{\infty} ({\hat x})=4\pi\sum_{n=0}^{+\infty}
\sum_{m=-n}^{n}(-{\rm i})^{n}(\alpha_n^{m}
+\eta_n^{m})Y_n^m({\hat x}). 
\end{align}
The proof is completed by applying Theorem \ref{The-123} and \eqref{plJA-4}.
\end{proof}

\subsection{Near-field data}

In this section, we explore the characterization of nonradiating source by using the near-field data, which involves the electric field $\boldsymbol E$ on $\partial B_R$.

Define 
\begin{align}\label{FTE0max-5}
 \boldsymbol{ U}(\xi)& =\int_{\partial B_R}\big[ \boldsymbol  n\times (\nabla\times\boldsymbol{E}(y))-\mathrm{i}{\xi}\times\big(\boldsymbol  n\times\big( \boldsymbol{E}(y)
 + \kappa^{-2}\nabla(\nabla\cdot{\boldsymbol  E}(y))\big)\big)\big]e^{-{\rm i} \xi \cdot y}{\rm d} s_y,\quad  |\xi|=\kappa.
\end{align}

\begin{theorem}\label{Themax-NDu}
The source $\boldsymbol{J}\in H(\mathrm{div}; \mathrm{grad}; B_R)$ is nonradiating if and only if   $\boldsymbol{ U}(\xi)=0$ for $|\xi|=\kappa$.
\end{theorem}

\begin{proof}
From \eqref{Eqal-2-2} and \eqref{EIN-gr}, for $|x|>|y|$, we obtain
\begin{align}\label{Elbmax}
&\int_{B_R}{\boldsymbol  G_l}(x, y)\cdot \boldsymbol{\mathcal{J}}(y){\rm d}y\nonumber\\
&=\int_{B_R}{\boldsymbol  G_l}(x, y)\cdot\Big[\nabla\times\nabla\times\big(
\boldsymbol{E}(y)+ \kappa^{-2}\nabla(\nabla\cdot{\boldsymbol  E}(y))\big)
-\kappa^2\big(\boldsymbol{E}(y)+ \kappa^{-2}\nabla(\nabla\cdot{\boldsymbol  E}(y))\big)\Big]{\rm d}y\nonumber\\
&=\int_{B_R} \bigg[(\nabla_y\times\nabla_y\times{\boldsymbol  G_l}(x, y)-\kappa^2{\boldsymbol  G_l}(x, y))\cdot\big(
\boldsymbol{E}(y)+ \kappa^{-2}\nabla(\nabla\cdot{\boldsymbol  E}(y))\big)
\Big]{\rm d}y\nonumber\\
&\quad+\int_{\partial B_R} \Big[\Big(\boldsymbol  n\times\big(
\boldsymbol{E}(y)+ \kappa^{-2}\nabla(\nabla\cdot{\boldsymbol  E}(y))\big)\Big)\cdot (\nabla_y\times{\boldsymbol  G_l}(x, y))\nonumber\\
&\quad\quad\quad\quad\quad -(\boldsymbol  n\times{\boldsymbol  G_l}(x, y))\cdot \Big(\nabla\times\big(
\boldsymbol{E}(y)+ \kappa^{-2}\nabla(\nabla\cdot{\boldsymbol  E}(y))\big)\Big)\Big]{\rm d} s_y\nonumber\\
&=\int_{\partial B_R} \Big[(\nabla_y\times{\boldsymbol  G_l}(x, y))\cdot \Big(\boldsymbol  n\times\big(
\boldsymbol{E}(y)+ \kappa^{-2}\nabla(\nabla\cdot{\boldsymbol  E}(y))\big)\Big)\nonumber\\
&\quad\quad\quad\quad-(\boldsymbol  n\times{\boldsymbol  G_l}(x, y))\cdot (\nabla\times
\boldsymbol{E}(y))\Big]{\rm d} s_y.  
\end{align}

Define 
\[
g(x)=\frac{e^{\mathrm{i}\kappa|x|}}{4\pi|x|}, \quad \mathcal{G}_l (x)=\big(\boldsymbol{e}_l+
\kappa^{-2}\nabla_{x}\partial_{ x_l}\big)g(x), \quad l=1,2,3, 
\]
where $\boldsymbol e_1=(1, 0, 0), \boldsymbol e_2=(0, 1, 0)$, and $\boldsymbol e_3=(0, 0, 1)$. Taking the Fourier transform and noting the convolution form on both sides of \eqref{Elbmax} gives
\begin{align}\label{FTEmax}
\mathcal{\hat G}_l(\xi)\cdot \boldsymbol {{\mathcal{\hat J}}}(\xi)
&={ \mathcal{\hat G}_l}(\xi)\cdot\int_{\partial B_R}\Big[
\boldsymbol  n\times (\nabla\times\boldsymbol{E}(y))\notag\\
& \quad -(\mathrm{i}{\xi})\times\Big(\boldsymbol  n\times\big(
\boldsymbol{E}(y)+ \kappa^{-2}\nabla(\nabla\cdot{\boldsymbol  E}(y))\big)\Big)\Big]e^{-{\rm i}\xi \cdot y}{\rm d} s_y, 
\end{align}
where 
\begin{align}\label{FTE0max-1}
\hat{\mathcal  G}_l(\xi)
&=\int_{\mathbb{R}^3}
\mathcal{ G}_l(x)e^{-{\rm i}\xi \cdot x}{\rm d} {x}
=\big(\boldsymbol e_l +\kappa^{-2}({\rm i}\xi_l)({\rm i}\xi)\big)\int_{\mathbb{R}^3}
g(x)e^{-{\rm i} \xi \cdot x}{\rm d} {x}\notag\\
&=\big(\boldsymbol e_l -\kappa^{-2}\xi_l \xi\big){\hat g}(\xi),\quad \xi \in \mathbb{R}^3,
\end{align}
where $\xi=(\xi_1, \xi_2, \xi_3)$. 

Let $\boldsymbol  W(\xi)=({ \mathcal{\hat G}_1}(\xi), {\mathcal{\hat G}_2}(\xi), {\mathcal{\hat G}_3}(\xi))$. We deduce from \eqref{FTEmax}--\eqref{FTE0max-1} that 
\begin{align}\label{FTE0max-2}
\boldsymbol  W(\xi)\boldsymbol {{\mathcal{\hat J}}}(\xi)
& =\boldsymbol  W(\xi)\int_{\partial B_R}\Big[\boldsymbol  n\times
(\nabla\times\boldsymbol{E}(y))\notag\\
&\quad -(\mathrm{i}{\xi})\times\Big(\boldsymbol  n\times\big(
\boldsymbol{E}(y)+ \kappa^{-2}\nabla(\nabla\cdot{\boldsymbol  E}(y))\big)\Big)\Big]e^{-{\rm i}\xi \cdot y}{\rm d} s_y,
\end{align}
where
\begin{align}\label{FTE0max-3}
{\boldsymbol W}(\xi)= \tilde{\boldsymbol W}(\xi){\hat g}(\xi)
=({\hat g}(\xi))^{\frac{2}{3}}\big[ \tilde{\boldsymbol W}(\xi)({\hat g}(\xi))^{\frac{1}{3}}\big]
\end{align}
and 
\[
 \tilde{\boldsymbol W}(\xi)=
 \begin{pmatrix}
1-\kappa^{-2}\xi_1^2 & -\kappa^{-2}\xi_1\xi_2&-\kappa^{-2}\xi_1\xi_3\\[5pt]
-\kappa^{-2}\xi_1\xi_2 & 1-\kappa^{-2}\xi_2^2&-\kappa^{-2}\xi_2\xi_3 \\[5pt]
-\kappa^{-2}\xi_1\xi_3 & -\kappa^{-2}\xi_2\xi_3&1-\kappa^{-2}\xi_3^2
 \end{pmatrix}.
\]

It can be verified that 
\begin{align}\label{FTE0max-4}
{\hat g}(\xi)
&=\int_{\mathbb{R}^3}
g(x)e^{-{\rm i}\xi \cdot x}{\rm d} {x}
=\int_{\mathbb{R}^3}
g(x)
\bigg[4\pi\sum_{n=0}^{+\infty}
\sum_{m=-n}^{n}
(-{{\rm i}})^{n}j_{n}(|{\xi}||x|)\overline{Y_n^m({\hat x})}Y_n^m({\hat\xi})\bigg]{\rm d} {x}\nonumber\\
&=4\pi\sum_{n=0}^{+\infty}
\sum_{m=-n}^{n}
(-{{\rm i}})^{n}Y_n^m({\hat\xi})\int_{0}^{+\infty}g(r) j_{n}(|{\xi}|r)r^2\bigg(\int_{0}^{\pi}\int_{0}^{2\pi} \overline{Y_n^m(\theta, \phi)}\sin\theta{\rm d} \theta{\rm d}\phi\bigg){\rm d}r\nonumber\\
&=4\pi\int_{0}^{+\infty}g(r) j_{0}(|\xi|r)r^2{\rm d}
={\rm i}\kappa \int_{0}^{+\infty}h_{0}^{(1)}(\kappa r) j_{0}(|\xi|r)r^2{\rm d}r\nonumber\\
&={\rm i}\kappa \bigg(\int_{0}^{+\infty}
j_{0}(\kappa r)j_{0}(|\xi|r)r^2{\rm d}r
+{\rm i}\int_{0}^{+\infty}
y_{0}(\kappa r)j_{0}(|\xi|r)r^2{\rm d}r\bigg)\nonumber\\
&={\rm i}\kappa \bigg(\frac{\pi}{2\kappa^2}\delta(\kappa-|\xi|)
+0\bigg)=\frac{{\rm i}\pi}{2\kappa}\delta(\kappa-|\xi|)
=\left\{
\begin{aligned}
&\infty, &\quad |\xi| = \kappa,\\
&0, &\quad |\xi| \neq \kappa. 
\end{aligned}
\right.
\end{align}
On the sphere $|\xi| = \kappa$, we obtain from \eqref{FTE0max-2}--\eqref{FTE0max-3} that   
\begin{align}\label{FTE0max-2-1}
&
\big[\tilde{\boldsymbol W}(\xi)({\hat g}(\xi))^{\frac{1}{3}}\big]
 \boldsymbol {{\mathcal{\hat J}}}(\xi)
=\big[\tilde{\boldsymbol{W}}(\xi)(g(\xi))^{\frac{1}{3}}\big]\notag\\
&\times\int_{\partial B_R}\Big[\boldsymbol  n\times
(\nabla\times\boldsymbol{E}(y))-
(\mathrm{i}{\xi})\times\Big(\boldsymbol  n\times\big(
\boldsymbol{E}(y)+ \kappa^{-2}\nabla(\nabla\cdot{\boldsymbol  E}(y))\big)\Big)\Big]e^{-{\rm i}\xi \cdot y}{\rm d} s_y.
\end{align}
A simple calculation yields 
\begin{align}\label{Wtdmax}
\det(\tilde{\boldsymbol W}(\xi))=1-\frac{|\xi|^2}{\kappa^2}=0,\quad \hat{g}(\xi)\simeq \frac{1}{\kappa-|\xi|}=\infty,\quad |\xi| = \kappa, 
\end{align}
which, together with \eqref{FTE0max-3} and \eqref{Wtdmax}, implies that 
\begin{align*}
\det\big[\tilde{\boldsymbol W}(\xi)(\hat{g}(\xi))^{\frac{1}{3}}\big]
=[(\hat{g}(\xi))^{\frac{1}{3}}]^3\det(\tilde{\boldsymbol W}(\xi))
=\hat{g}(\xi)\det(\tilde{\boldsymbol W}(\xi))
\simeq\kappa+|\xi|\neq 0,\quad |\xi|=\kappa,
\end{align*}
Hence, the matrix $\big[\tilde{\boldsymbol W}(\xi)({\hat g}(\xi))^{\frac{1}{3}}\big]$ is invertible for $|\xi|=\kappa$. It follows from \eqref{FTE0max-2-1} that
\begin{align}\label{FTE0max-5}
\boldsymbol {{\mathcal{\hat J}}}(\xi)
 &=\int_{\partial B_R}\Big[
\boldsymbol  n\times
(\nabla\times\boldsymbol{E}(y))-
(\mathrm{i}{\xi})\times\Big(\boldsymbol  n\times\big(
\boldsymbol{E}(y)+ \kappa^{-2}\nabla(\nabla\cdot{\boldsymbol  E}(y))\big)\Big)\Big]e^{-{\rm i}\xi \cdot y}{\rm d} s_y\nonumber\\
&=\boldsymbol{U}(\xi),\quad |\xi|=\kappa.
\end{align}
The conclusion of the theorem now follows from  Theorem \ref{The-FE} and \eqref{FTE0max-5}.
\end{proof}

Let
\begin{align}\label{FTE0-5}
 \boldsymbol{V}(\xi)
 =\int_{\partial B_R}\big[
\boldsymbol{E}(y)(\mathrm{i}{\xi}\cdot\boldsymbol  n(y))
-(\nabla\boldsymbol{E}(y)\boldsymbol  n(y))\big]e^{-{\rm i}\xi \cdot y}{\rm d}s_y,\quad |\xi|=\kappa.
\end{align}

\begin{theorem}\label{TheHem-NDu}
The source $\boldsymbol{J}\in H(\mathrm{div}; \mathrm{grad}; B_R)$ is nonradiating if and only if   $\boldsymbol{V}(\xi)=0$ for $|\xi|=\kappa$.
\end{theorem}

\begin{proof}
From \eqref{Eqal-2-3} and \eqref{EIN-gr1}, for $|x|>|y|$, we obtain
\begin{align}\label{Elb}
&\int_{B_R} g(x, y) \boldsymbol{\mathcal{J}}(y){\rm d}y
=\int_{B_R} g(x, y) \big(\boldsymbol  J (y)+ \kappa^{-2}\nabla(\nabla\cdot{\boldsymbol  J}(y))\big){\rm d} y\nonumber\\
&=\int_{B_R} g(x, y)\big(\boldsymbol  J (y)+ \kappa^{-2}\nabla(\nabla\cdot{\boldsymbol  J}(y))\big){\rm d}y
=\int_{B_R} g(x, y) \big[-(\Delta\boldsymbol{E}(y)
+\kappa^2\boldsymbol{E}(y))\big]{\rm d}y\nonumber\\
&=-\int_{B_R} \big[\big(\Delta_y g(x, y)+\kappa^2 g(x, y)\big)\boldsymbol{E}(y)\big]{\rm d}y\nonumber\\
&\quad-\int_{\partial B_R} \big[g(x, y)(\nabla\boldsymbol{E}(y)\boldsymbol  n(y))-\boldsymbol{E}(y)(\nabla_y g(x, y)\cdot\boldsymbol  n(y))\big]{\rm d} s_y\nonumber\\
&=\int_{\partial B_R}
\big[\boldsymbol{E}(y)(\nabla_y g(x, y)\cdot\boldsymbol  n(y))-g(x, y)(\nabla\boldsymbol{E}(y)\boldsymbol  n(y))\big]{\rm d} s_y. 
\end{align}
Taking the Fourier transform on both sides of \eqref{Elb} leads to 
\begin{align}\label{FTE}
{\hat g}(\xi) \boldsymbol {{\mathcal{\hat J}}}(\xi)
&={\hat g}(\xi)\int_{\partial B_R}\big[
\boldsymbol{E}(y)(\mathrm{i}{\xi}\cdot\boldsymbol  n(y))
- (\nabla\boldsymbol{E}(y)\boldsymbol  n(y))\big]e^{-{\rm i}\xi \cdot y}{\rm d} s_y.
\end{align}
We obtain from \eqref{FTE0max-4} and \eqref{FTE} that 
\begin{align}\label{FTE0-5}
 \boldsymbol {{\mathcal{\hat J}}}(\xi)
 &=\int_{\partial B_R}\big[
\boldsymbol{E}(y)(\mathrm{i}{\xi}\cdot\boldsymbol  n(y))
-(\nabla\boldsymbol{E}(y)\boldsymbol  n(y))\big]e^{-{\rm i}\xi \cdot y}{\rm d}s_y
=\boldsymbol{V}(\xi),\quad |\xi|=\kappa.
\end{align}
The proof is completed by combining Theorem \ref{The-FE} and \eqref{FTE0-5}.
\end{proof}

\subsection{Null spaces}\label{ns}

This section aims to investigate the characterizations of nonradiating sources in Maxwell's equations through the analysis of null spaces associated with integral operators.

Consider the complex conjugate of the Green's tensor for the Maxwell equation
\begin{align*}
\boldsymbol{G}^{*}(x, y)=\big(\boldsymbol{I}+\kappa^{-2}
\nabla_x\nabla_x\big)\overline{g(x, y)},
\end{align*}
which satisfies
\begin{equation*}
\nabla_x\times(\nabla_x\times\boldsymbol{G}^{*}
(x, y))-\kappa^{2}\boldsymbol{G}^{*}(x, y)
=\delta(x-y)\boldsymbol{I}.
\end{equation*}
A simple calculation yields
\begin{align*}
\boldsymbol{G}(x, y)-\boldsymbol{G}^{*}(x, y)
&=\big(\boldsymbol{I}+\kappa^{-2}
\nabla_x\nabla_x\big)
\bigg(\frac{e^{\mathrm{i}\kappa|x-y|}}{
4\pi|x -y|}-\frac{e^{-\mathrm{i}\kappa|x-y|}}{
4\pi|x -y|}\bigg)\nonumber\\
&=\big(\boldsymbol{I}+\kappa^{-2}
\nabla_x\nabla_x\big)\bigg(\frac{2\mathrm{i}\sin{(\kappa|x-y|)}}{
4\pi|x -y|}\bigg)
=\frac{\mathrm{i}\kappa}{
2\pi}\big(\boldsymbol{I}+\kappa^{-2}
\nabla_y\nabla_y\big)j_0(\kappa|x-y|),
\end{align*}
where $j_0$ is the spherical Bessel function of order 0. 

\begin{definition}\label{Nsde1}
The source function $\boldsymbol  J\in (L^2(B_R))^3$ is considered to be in the null space $\mathcal{N}_1(R)$ if
\begin{align}\label{DNspa-1}
\int_{B_R}\Big[\big(\boldsymbol{I}+\kappa^{-2}\nabla_y\nabla_y\big)j_0(\kappa|x-y|)\Big]\boldsymbol{J}(y){\rm d}y =0,\quad |x| > R.
\end{align}
\end{definition}

For $x, y\in\mathbb{R}^3$, it is noteworthy that $j_{0}$ has a spherical harmonic expansion given by
\begin{align}\label{j0ex}
j_{0}(\kappa|x-y|)
=4\pi\sum_{n=0}^{\infty}\sum_{m=-n}^{n}j_{n}(\kappa|x|)Y_{n}^{m}({\hat x})j_{n}(\kappa|y|)\overline{Y_{n}^{m}({\hat y})},\quad |x|>|y|.
\end{align}
We have from \eqref{DNspa-1}--\eqref{j0ex} that
\begin{align}\label{BRJbe-1}
&\int_{B_R}\Big[\big(\boldsymbol{I}+\kappa^{-2}
\nabla_y\nabla_y\big)j_0(\kappa|x-y|)\Big]\boldsymbol{J}(y){\rm d}y \nonumber\\
&=\sum_{n=0}^{\infty}\sum_{m=-n}^{n}4\pi j_{n}(\kappa|x|)Y_{n}^{m}({\hat x})\int_{B_R}\bigg(\big(\boldsymbol{I}+\kappa^{-2}\nabla_y\nabla_y \big)j_{n}(\kappa|y|)\overline{Y_{n}^{m}({\hat y})}\bigg) \boldsymbol{J}(y){\rm d}  y\nonumber\\
&=4\pi\sum\limits_{n=0}^{\infty}\sum\limits_{m=-n}^{n} j_{n}(\kappa|x|)Y_{n}^{m}({\hat x})\beta_n^{m},
\quad |x|>R.
\end{align}

Building upon Theorem \ref{The-123}, we can deduce the following result.

\begin{theorem}\label{The-NR-1}
Assuming $\boldsymbol{J} \in (L^2(B_R))^3$, the source $\boldsymbol{J}$ is nonradiating if and only if $\boldsymbol{J} \in \mathcal{N}_1(R)$.
\end{theorem}

\begin{proof}
If $\boldsymbol{J} \in (L^2(B_R))^3$ is a nonradiating source for Maxwell's equations, Theorem \ref{The-123}, together with \eqref{BRJbe-1}, implies that $\boldsymbol{J}$ belongs to the null space $\mathcal{N}_1(R)$.

On the other hand, if $\boldsymbol{J} \in \mathcal{N}_1(R)$, then, from \eqref{BRJbe-1}, for all $|x|>R$, we have
\begin{align*}
0&=\int{B_R}\left[\left(\boldsymbol{I}+\kappa^{-2}\nabla_y\nabla_y\right)j_0(\kappa|x-y|)\right]\boldsymbol{J}(y)dy\\
&=4\pi\sum_{n=0}^{\infty}\sum_{m=-n}^{n}
j_{n}(\kappa|x|)Y_{n}^{m}(\hat{x})\beta_n^{m}, \quad|x|>R.
\end{align*}
Hence, we obtain $\beta_n^m=0$ for all $n = 0,1,\dots, m = -n, \dots, n$. It follows from Theorem \ref{The-123} that $\boldsymbol{J}$ is a nonradiating source.
\end{proof}

\begin{definition}\label{Nsde2}
The source function $\boldsymbol  J\in H_0(\mathrm{div};B_R)$ is considered to be in the null space $\mathcal{N}_2(R)$ if
\begin{align}\label{DNspa-2}
\int_{B_R} j_0(\kappa|x-y|)\boldsymbol{J}(y){\rm d}y -\kappa^{-2} \int_{B_R}\nabla_y j_0(\kappa|x-y|)\nabla\cdot\boldsymbol{J}(y) {\rm d}y=0,\quad |x| > R.
\end{align}
\end{definition}

\begin{theorem}\label{The-NR-2}
Assuming $\boldsymbol  J\in H_0(\mathrm{div};B_R)$, the source $\boldsymbol  J$ is nonradiating if and only if $\boldsymbol  J \in \mathcal{N}_2(R)$.
\end{theorem}

\begin{proof}
For $\boldsymbol  J\in H_0(\mathrm{div};B_R)$, we deduce from \eqref{j0ex} and the divergence theorem that 
\begin{align}\label{BRJbe-2}
&\int_{B_R} j_0(\kappa|x-y|)\boldsymbol{J}(y){\rm d}y - \kappa^{-2}\int_{B_R}\nabla_y j_0(\kappa|x-y|)\nabla\cdot\boldsymbol{J}(y){\rm d}y \nonumber\\
&=4\pi \sum_{n=0}^{\infty}\sum_{m=-n}^{n}j_{n}(\kappa|x|)Y_{n}^{m}({\hat x})
\bigg[\int_{B_R} j_{n}(\kappa|y|)\overline{Y_{n}^{m}({\hat y})} \boldsymbol{J}(y){\rm d}y
-\kappa^{-2}\int_{B_R}\Big(\nabla_y j_{n}(\kappa|y|)\overline{Y_{n}^{m}({\hat y})}\Big) \nabla\cdot\boldsymbol{J}(y){\rm d}y\bigg]\nonumber\\
&=4\pi\sum\limits_{n=0}^{\infty}\sum\limits_{m=-n}^{n}
 j_{n}(\kappa|x|)Y_{n}^{m}({\hat x})(\alpha_n^{m}
+\zeta_n^{m}), \quad |x|>R.
\end{align}
The conclusion of the theorem now follows from Theorem \ref{The-123} and \eqref{BRJbe-2}.
\end{proof}

\begin{definition}\label{Nsde3}
The source function $\boldsymbol  J\in H(\mathrm{div}; \mathrm{grad}; B_R)$ is considered to be in the null space $\mathcal{N}_3(R)$ if
\begin{align}\label{DNspa-3}
\int_{B_R}j_0(\kappa|x-y|)\Big(\boldsymbol  J (y)+ \kappa^{-2}\nabla(\nabla\cdot{\boldsymbol  J}(y))\bigg){\rm d}  y =0,\quad |x| > R.
\end{align}
\end{definition}

\begin{theorem}\label{The-NR-3}
Assuming $\boldsymbol  J\in H(\mathrm{div}; \mathrm{grad}; B_R)$, the source $\boldsymbol  J$ is nonradiating if and only if $\boldsymbol  J \in \mathcal{N}_3(R)$.
\end{theorem}

\begin{proof}
For $\boldsymbol  J\in H(\mathrm{div}; \mathrm{grad}; B_R)$, we deduce from \eqref{j0ex} that 
\begin{align}\label{BRJbe-3}
&\int_{B_R}j_0(\kappa|x-y|)\Big(\boldsymbol  J (y)+ \kappa^{-2}\nabla(\nabla\cdot{\boldsymbol  J}(y))\Big){\rm d}y \nonumber\\
&=4\pi \sum_{n=0}^{\infty}\sum_{m=-n}^{n}j_{n}(\kappa|x|)Y_{n}^{m}({\hat x})\int_{B_R}(j_{n}(\kappa|y|)\overline{Y_{n}^{m}({\hat y})}) \Big(\boldsymbol  J (y)+ \kappa^{-2}\nabla(\nabla\cdot{\boldsymbol  J}(y))\Big){\rm d}y\nonumber\\
&=4\pi\sum\limits_{n=0}^{\infty}\sum\limits_{m=-n}^{n}
 j_{n}(\kappa|x|)Y_{n}^{m}({\hat x})
(\alpha_n^{m} +\eta_n^{m}), \quad |x|>R.
\end{align}
The proof is completed by combining Theorem \ref{The-123} and \eqref{BRJbe-3}.
\end{proof}

\section{Examples of nonradiating sources}\label{S:ex}

In this section, we provide several specific examples of nonradiating sources to illustrate that the null spaces discussed in Section \ref{ns} are nontrivial.

\subsection{General nonradiating sources}

Consider a smooth function with compact support $\boldsymbol  F\in (C_{0}^{\infty}(D))^3$, where $D$ represents an open domain of compact support contained in the ball $B_R$. By applying the Maxwell operator $(\nabla\times\nabla\times-\kappa^2)$ to the function $\boldsymbol  F$, resulting in the function $\boldsymbol  J$, we obtain
\begin{equation}\label{JxD}
\boldsymbol  J(x)=(\nabla\times\nabla\times-\kappa^2)\boldsymbol  F(x)\in (C_{0}^{\infty}(D))^3,\quad x\in\mathbb{R}^3.
\end{equation}

By Definition \ref{def-E}, it is evident that the function $\boldsymbol{J}$ given in \eqref{JxD} also belongs to the space $(C_{0}^{\infty}(D))^3$ and constitutes a nonradiating source for Maxwell's equations. From \eqref{Eqal-3}, \eqref{EIN} and \eqref{JxD}, we have from integration by parts that 
\begin{align*} 
\boldsymbol  F(x)&=\int_{B_R}{\boldsymbol  G}(x, y) \boldsymbol{J}(y){\rm d} y
=\int_{B_R}\boldsymbol  G(x, y) [\nabla\times\nabla\times\boldsymbol{F}(y)
-\kappa^2\boldsymbol{F}(y)]{\rm d}y\nonumber\\
&=\int_{B_R} [(\nabla_{y}\times\nabla_{y}\times{\boldsymbol  G}(x, y)
-\kappa^2{\boldsymbol  G}(x, y))\boldsymbol{F}(y)]{\rm d}y\nonumber\\
&\quad+\int_{\partial B_R} [(\nabla_{y}\times{\boldsymbol  G}(x, y))(\boldsymbol  n\times\boldsymbol{F}(y))-(\boldsymbol  n\times{\boldsymbol  G}(x, y)) (\nabla\times\boldsymbol{F}(y))]{\rm d} s_y=0,\quad |x|>R,
\end{align*}
which implies that $\boldsymbol  J$ is a nonradiating source.

Below, we demonstrate, through equivalent characterizations, that the function $\boldsymbol{J}$ is a nonradiating source. It is clear to note that $\boldsymbol  J(x)\in H(\mathrm{div}; \mathrm{grad}; B_R)$ and
\begin{align}\label{Jhm1-1-1}
\boldsymbol  J(x)+\kappa^{-2}\nabla(\nabla\cdot\boldsymbol  J(x))
=-(\Delta\boldsymbol{F}(x)+\kappa^2\boldsymbol{F}(x))\in (C_{0}^{\infty}(D))^3. 
\end{align}
For all $n = 0,1,\dots, m = -n, \dots, n$, we deduce from \eqref{Jhm1-1-1} and integration by parts that 
\begin{align*}
\alpha_n^{m} +\eta_n^{m}&=\int_{B_R}j_{n}(\kappa|y|)\overline{Y_{n}^{m}({\hat y})}\Big({\boldsymbol  J}(y)+\kappa^{-2}\nabla(\nabla\cdot{\boldsymbol  J}(y))\Big){\rm d}y\nonumber\\
&=-\int_{B_R}j_{n}(\kappa|y|)\overline{Y_{n}^{m}({\hat y})}
(\Delta\boldsymbol{F}(y)+\kappa^2\boldsymbol{F}(y)){\rm d}y
\nonumber\\
&=-\int_{B_R}\big[(\Delta+\kappa^2)j_{n}(\kappa|y|)\overline{Y_{n}^{m}({\hat y})}\big]
\boldsymbol{F}(y){\rm d}y=0,
\end{align*}
which implies that $\boldsymbol  J \in \mathcal{N}_3(R)$.

\subsection{A nonradiating source in $\mathcal{N}_1(R)$}\label{exn2}

Define the space $H(\mathrm{grad};B_R)$ by 
\begin{align*}
H(\mathrm{grad};B_R)=\big\{{\Psi}\in L^2(B_R): \nabla{\Psi}\in (L^2(B_R))^3\big\},
\end{align*} 
and its subspace 
\begin{align*}H_0(\mathrm{grad};B_R)=
\big\{{\Psi}\in H(\mathrm{grad};B_R): {\Psi}|_{\partial B_R}=0\big\}.
\end{align*}

Consider a function $Q\in H_{0}(\mathrm{grad}; B_R)$,  which is assumed to have a compact support contained in the ball $B_R$. Let $\boldsymbol  J(x)=\nabla  Q(x)$, then, $\boldsymbol  J\in (L^2(B_R))^3$. It follows from straightforward calculations that 
\begin{align}\label{JgQ}
\beta_{n}^{m}
&=\int_{B_R}\big[j_{n}(\kappa|y|)\overline{Y_{n}^{m}({\hat y})}\boldsymbol{J}(y)\big]{\rm d}y
+\kappa^{-2}\int_{B_R}\big[\big(\nabla_y\nabla_y j_{n}(\kappa|y|)\overline{Y_{n}^{m}({\hat y})} \big)\boldsymbol{J}(y)\big]{\rm d}y\nonumber\\
&=\int_{B_R}\big[j_{n}(\kappa|y|)\overline{Y_{n}^{m}({\hat y})}\boldsymbol{J}(y)\big]{\rm d}y
-\kappa^{-2}\int_{B_R}\big[\big(\Delta_y\nabla_y j_{n}(\kappa|y|)\overline{Y_{n}^{m}({\hat y})} \big)Q(y)\big]{\rm d} y\nonumber\\
&\quad+\kappa^{-2}\int_{\partial B_R}\big[\big(\boldsymbol  n(y)\cdot\nabla_y\nabla_y j_{n}(\kappa|y|)\overline{Y_{n}^{m}({\hat y})} \big)Q(y)\big]{\rm d}y\nonumber\\
&=\int_{B_R}\big[j_{n}(\kappa|y|)\overline{Y_{n}^{m}({\hat y})}\boldsymbol{J}(y)\big]{\rm d}y
-\kappa^{-2}\int_{B_R}\big[\big(\nabla_y\Delta_y j_{n}(\kappa|y|)\overline{Y_{n}^{m}({\hat y})} \big)Q(y)\big]{\rm d}y\nonumber\\
&=\int_{B_R}\big[j_{n}(\kappa|y|)\overline{Y_{n}^{m}({\hat y})}\boldsymbol{J}(y)\big]{\rm d}y
+\kappa^{-2}\int_{B_R}\big[\big(\Delta_y j_{n}(\kappa|y|)\overline{Y_{n}^{m}({\hat y})} \big)(\nabla Q(y))\big]{\rm d}y\nonumber\\
&=\int_{B_R}\big[j_{n}(\kappa|y|)\overline{Y_{n}^{m}({\hat y})}\boldsymbol{J}(y)\big]{\rm d}y
-\int_{B_R}\big[\big( j_{n}(\kappa|y|)\overline{Y_{n}^{m}({\hat y})} \big)(\nabla Q(y))\big]{\rm d}y
=0,
\end{align}
which implies that $\boldsymbol  J \in \mathcal{N}_1(R)$.

\subsection{A nonradiating source in $\mathcal{N}_2(R)$}

Choose $\kappa>0$ and $R>0$ such that $\kappa R$ is the first root of the Spherical Bessel function of order 0, i.e., $j_0(\kappa R)=0$. Define 
\begin{align}\label{CJm12}
\boldsymbol  J(x)
=\left\{
\begin{aligned}
&\frac{\nabla j_0^{m_1}(\kappa |x|)}{N^{m_1}}-\frac{\nabla j_0^{m_2}(\kappa |x|)}{N^{m_2}}, &\quad |x|<R,\\
&0, &\quad |x|\geq R,
\end{aligned}
\right.
\end{align}
where $m_1, m_2>2$ are integers, $m_1\neq m_2$, and 
\[
N^{m_l}=m_l\int_{0}^{R}j_{0}^{m_l-1}(\kappa r)j_1^2(\kappa r)r^2{\rm d}r,  \quad l=1, 2. 
\]
Evidently, we have $\frac{ j_0^{m_1}(\kappa |x|)}{N^{m_1}}-\frac{ j_0^{m_2}(\kappa |x|)}{N^{m_2}}\in H_0(\mathrm{grad};B_R)$. It is clear to note from Section \ref{exn2} that the function $\boldsymbol  J$ given by \eqref{CJm12} is a nonradiating source. Furthermore, for $|x|<R$, we have from straightforward calculations that 
\begin{align}\label{CJm12-1}
\nabla j_0^{m_l}(\kappa |x|)
=\kappa m_lj_0^{m_l-1}(\kappa |x|)j_0^{'}(\kappa |x|){\hat{x}}
=-\kappa m_l j_0^{m_l-1}(\kappa |x|)j_1(\kappa |x|){\hat{x}}
\neq 0,
\end{align}
and
\begin{align}\label{Jgrad-3}
&\nabla\cdot{\boldsymbol  J}(\kappa |x|)
=\sum_{s=1}^3\frac{\partial\big[- \big(\frac{ m_1 j_0^{m_1-1}(\kappa |x|)}{N^{m_1}}
-\frac{m_2 j_0^{m_2-1}(\kappa |x|)}{N^{m_2}}\big)j_1(\kappa |x|)\frac{\kappa x_s}{ |x|}\big]}{\partial x_s}\nonumber\\
&=\frac{\kappa^2 \big[m_1(m_1-1)j_0^{m_1-2}(\kappa |x|)j_1^{2}(\kappa |x|)
+ m_1 j_0^{m_1}(\kappa |x|)
-4 m_1(\kappa |x|)^{-1}j_0^{m_1-1}(\kappa |x|)j_1(\kappa |x|)\big]}{N^{m_1}}
\nonumber\\
&\quad-\frac{\kappa^2 \big[m_2(m_2-1)j_0^{m_2-2}(\kappa |x|)j_1^{2}(\kappa |x|)
+ m_2 j_0^{m_2}(\kappa |x|)
-4 m_2(\kappa |x|)^{-1}j_0^{m_2-1}(\kappa |x|)j_1(\kappa |x|)\big]}{N^{m_2}}\nonumber\\
&\neq 0.
\end{align}
Hence, it follows from \eqref{CJm12} and \eqref{Jgrad-3} that $\boldsymbol  J\in H_0(\mathrm{div};B_R)$. Then, we have
\begin{align*}
&\int_{B_R}{\boldsymbol  G}(x, y)\boldsymbol{J}(y){\rm d}y
=\mathrm{i}\kappa\sum_{n=0}^{\infty}\sum_{m=-n}^{n} h_{n}^{(1)}(\kappa|x|)Y_{n}^{m}({\hat x})\int_{B_R}\bigg[\Big(\big(\boldsymbol{I}+\kappa^{-2}
\nabla_y\nabla_y\big)j_{n}(\kappa|y|)\overline{Y_{n}^{m}({\hat y})} \Big)\boldsymbol{J}(y)\bigg]{\rm d}y\nonumber\\
&=\mathrm{i}\kappa\sum_{n=0}^{\infty}\sum_{m=-n}^{n} h_{n}^{(1)}(\kappa|x|)Y_{n}^{m}({\hat x})\nonumber\\
&\quad\quad\quad\quad\quad\quad\quad\times\bigg[\int_{B_R}j_{n}(\kappa|y|)\overline{Y_{n}^{m}({\hat y})} \boldsymbol{J}(y){\rm d}y -\frac{1}{\kappa^2}\int_{B_R}\Big(\nabla_{y}j_{n}(\kappa|y|)\overline{Y_{n}^{m}({\hat y})}\Big) \nabla\cdot\boldsymbol{J}(y){\rm d}y\bigg]\nonumber\\
&=\mathrm{i}\kappa\sum_{n=0}^{\infty}\sum_{m=-n}^{n} h_{n}^{(1)}(\kappa|x|)Y_{n}^{m}({\hat x})(\alpha_n^{m}+  \zeta_n^{m}), \quad |x|>R,
\end{align*}
which implies that $\boldsymbol  J\in \mathcal{N}_2(R)$ and $\alpha_n^{m}
+\zeta_n^{m}=0$ for all $n = 0,1,\dots, m = -n, \dots, n$.

Next, we prove that $\alpha_n^{m}
=\zeta_n^{m}=0$ for all $n = 0,1,\dots, m = -n, \dots, n$. Explicitly, the spherical harmonic functions $Y_{n}^{m}$ are given by  
\begin{align}\label{CJm12-5}
Y_{n}^{m}(\theta,\phi)=N_n^m P_{n}^{|m|}(\cos\theta) e^{{\rm i}m\phi},
\end{align}
where 
\[
N_n^m=\sqrt{\frac{(2n+1)(n-|m|)!}{4\pi(n+|m|)!}}, 
\]
and $P_{n}^{|m|}(t)$ are the associated Legendre functions and are
defined by
\begin{align}\label{CJm12-6}
P_{n}^{|m|}(t)
=(1-t^2)^{\frac{|m|}{2}}\frac{{\rm d}^{|m|}}{{\rm d} t^{|m|}}P_{n}(t), \quad -n \leq m \leq n, \quad |t| \leq 1.
\end{align}
Here $P_{n}(t)$ is the Legendre polynomial of degree $n$. Combining \eqref{CJm12}, \eqref{CJm12-1}, \eqref{CJm12-5}
and \eqref{CJm12-6}, we discuss the following two cases:

Case 1. $n=0,1,\cdots;$ $m=-n,\cdots,n$ and $m\neq 0,\pm 1$.

\begin{align}\label{NQ-0}
\alpha_{n}^{m}&=\int_{B_R}j_{n}(\kappa|y|)\overline{Y_{n}^{m}({\hat y})} \bigg(\frac{-m_1\kappa j_0^{m_1-1}(\kappa |y|)j_1(\kappa |y|){\hat{y}}}{N^{m_1}}-\frac{-m_2\kappa j_0^{m_2-1}(\kappa |y|)j_1(\kappa |y|){\hat{y}}}{N^{m_2}}\bigg)
{\rm d}y\nonumber\\
&=-N_n^m \kappa\bigg(\frac{m_1 \int_{0}^{R}j_n(\kappa r)j_{0}^{m_1-1}(\kappa r)j_1(\kappa r)r^2{\rm d}r}{N^{m_1}}
-\frac{m_2 \int_{0}^{R}j_n(\kappa r)j_{0}^{m_2-1}(\kappa r)j_1(\kappa r)r^2{\rm d}r}{N^{m_2}}
\bigg)\times\nonumber\\
&\quad\quad\bigg(\int_{0}^{\pi}\int_{0}^{2\pi}  \overline{P_{n}^{|m|}(\cos\theta) e^{{\rm i}m\phi}}\sin\theta\begin{pmatrix}
\sin\theta\cos\phi\\
\sin\theta\sin\phi\\
\cos\theta
\end{pmatrix}
{\rm d}\phi{\rm d}\theta\bigg)\nonumber\\
&=-N_n^m \kappa\bigg(\frac{m_1 \int_{0}^{R}j_n(\kappa r)j_{0}^{m_1-1}(\kappa r)j_1(\kappa r)r^2{\rm d}r}{N^{m_1}}
-
\frac{m_2 \int_{0}^{R}j_n(\kappa r)j_{0}^{m_2-1}(\kappa r)j_1(\kappa r)r^2{\rm d}r}{N^{m_2}}
\bigg)\times\nonumber\\
&\quad\quad \begin{pmatrix}
\big(\int_{0}^{\pi}P_{n}^{|m|}(\cos\theta)\sin^2\theta{\rm d}\theta\big)\times 0\\
\big(\int_{0}^{\pi}P_{n}^{|m|}(\cos\theta)\sin^2\theta{\rm d}\theta\big)\times 0\\
\big(\int_{0}^{\pi}P_{n}^{|m|}(\cos\theta)\sin\theta\cos\theta{\rm d}\theta\big)\times 0
\end{pmatrix}\nonumber\\
&=-N_n^m \kappa\bigg(\frac{m_1 \int_{0}^{R}j_n(\kappa r)j_{0}^{m_1-1}(\kappa r)j_1(\kappa r)r^2{\rm d}r}{N^{m_1}}
-
\frac{m_2 \int_{0}^{R}j_n(\kappa r)j_{0}^{m_2-1}(\kappa r)j_1(\kappa r)r^2{\rm d}r}{N^{m_2}}
\bigg)\times \begin{pmatrix}
0\\
0\\
0
\end{pmatrix}
\nonumber\\
&=0.
\end{align}

Case 2. $n=0,1,\cdots;$ $m=0, \pm1$.

\begin{align}\label{NQ-1}
\alpha_{n}^{0}
&=-N_n^0 \kappa\bigg(\frac{m_1 \int_{0}^{R}j_n(\kappa r)j_{0}^{m_1-1}(\kappa r)j_1(\kappa r)r^2{\rm d}r}{N^{m_1}}
-\frac{m_2 \int_{0}^{R}j_n(\kappa r)j_{0}^{m_2-1}(\kappa r)j_1(\kappa r)r^2{\rm d}r}{N^{m_2}}
\bigg)\times\nonumber\\
&\quad\quad\bigg(\int_{0}^{\pi}\int_{0}^{2\pi}  \overline{P_{n}^{0}(\cos\theta) e^{\mathbf{i}0\phi}}\sin\theta\begin{pmatrix}
\sin\theta\cos\phi\\
\sin\theta\sin\phi\\
\cos\theta
\end{pmatrix}
{\rm d}\phi{\rm d}\theta\bigg)\nonumber\\
&=-N_n^0 \kappa\bigg(\frac{m_1 \int_{0}^{R}j_n(\kappa r)j_{0}^{m_1-1}(\kappa r)j_1(\kappa r)r^2{\rm d}r}{N^{m_1}}
-
\frac{m_2 \int_{0}^{R}j_n(\kappa r)j_{0}^{m_2-1}(\kappa r)j_1(\kappa r)r^2{\rm d}r}{N^{m_2}}
\bigg)\times\nonumber\\
&\quad\quad \begin{pmatrix}
0\\
0\\
2\pi\big(\int_{-1}^{1}P_{n}(t)P_{1}(t){\rm d}t\big)
\end{pmatrix}.
\end{align}
Noting
\begin{align}\label{NQ-2}
\int_{-1}^{1}P_{n}(t)P_{1}(t){\rm d}t
=
\left\{\begin{array}{lll}
	\frac{2}{2+1}, && n=1,\vspace{3pt}\\
	0,&& n\neq1.
\end{array}\right.
\end{align}
 we obtain from  \eqref{NQ-1} and \eqref{NQ-2} that 
\begin{align}\label{NQ-3}
\alpha_{n}^{0}
&=\left\{\begin{array}{lll}
	-N_n^0 \kappa\bigg(\frac{N^{m_1}}{N^{m_1}}
-
\frac{N^{m_2}}{N^{m_2}}
\bigg)\times \begin{pmatrix}
0\\
0\\
\frac{4\pi}{3}
\end{pmatrix}, && n=1,\vspace{3pt}\\
\begin{pmatrix}
0\\
0\\
0
\end{pmatrix},&& n\neq1,
\end{array}\right.
=0.
\end{align}

Furthermore,
\begin{align}\label{NQ-4}
\alpha_{n}^{\pm1}
&=-N_n^{\pm1} \kappa\bigg(\frac{m_1 \int_{0}^{R}j_n(\kappa r)j_{0}^{m_1-1}(\kappa r)j_1(\kappa r)r^2{\rm d}r}{N^{m_1}}
-
\frac{m_2 \int_{0}^{R}j_n(\kappa r)j_{0}^{m_2-1}(\kappa r)j_1(\kappa r)r^2{\rm d}r}{N^{m_2}}
\bigg)\times\nonumber\\
&\quad\quad\bigg(\int_{0}^{\pi}\int_{0}^{2\pi}  \overline{P_{n}^{\pm1}(\cos\theta) e^{\pm{\rm i}\phi}}\sin\theta\begin{pmatrix}
\sin\theta\cos\phi\\
\sin\theta\sin\phi\\
\cos\theta
\end{pmatrix}
{\rm d}\phi{\rm d}\theta\bigg)\nonumber\\
&=-N_n^{\pm1} \kappa\bigg(\frac{m_1 \int_{0}^{R}j_n(\kappa r)j_{0}^{m_1-1}(\kappa r)j_1(\kappa r)r^2{\rm d}r}{N^{m_1}}
-
\frac{m_2 \int_{0}^{R}j_n(\kappa r)j_{0}^{m_2-1}(\kappa r)j_1(\kappa r)r^2{\rm d}r}{N^{m_2}}
\bigg)\times\nonumber\\
&\quad\quad \begin{pmatrix}
\big(\int_{0}^{\pi}P_{n}^{\pm1}(\cos\theta)\sin^2\theta{\rm d}\theta\big)\big(\int_{0}^{2\pi} \cos\phi e^{\mp{\rm i}\phi} {\rm d}\phi\big)\\
\big(\int_{0}^{\pi}P_{n}^{\pm1}(\cos\theta)\sin^2\theta{\rm d}\theta\big)\big(\int_{0}^{2\pi} \sin\phi e^{\mp{\rm i}\phi} {\rm d}\phi\big)\\
\big(\int_{0}^{\pi}P_{n}^{\pm1}(\cos\theta)\sin\theta\cos\theta{\rm d}\theta\big)\big(\int_{0}^{2\pi}e^{\mp{\rm i}\phi}{\rm d}\phi\big)
\end{pmatrix}\nonumber\\
&=-N_n^{\pm1} \kappa\bigg(\frac{m_1 \int_{0}^{R}j_n(\kappa r)j_{0}^{m_1-1}(\kappa r)j_1(\kappa r)r^2{\rm d}r}{N^{m_1}}
-
\frac{m_2 \int_{0}^{R}j_n(\kappa r)j_{0}^{m_2-1}(\kappa r)j_1(\kappa r)r^2{\rm d}r}{N^{m_2}}
\bigg)\times\nonumber\\
&\quad\quad \begin{pmatrix}
\big(\int_{0}^{\pi}P_{n}^{\pm1}(\cos\theta)\sin^2\theta{\rm d}\theta\big)\big(\int_{0}^{2\pi} \cos^2\phi  {\rm d}\phi\big)\\
\mp{\rm i}\big(\int_{0}^{\pi}P_{n}^{\pm1}(\cos\theta)\sin^2\theta{\rm d}\theta\big)\big(\int_{0}^{2\pi} \sin^2\phi {\rm d}\phi\big)\\
 0
\end{pmatrix}.
\end{align}

Noting that 
\begin{align}\label{NQ-5}
&\int_{0}^{\pi}P_{n}^{\pm1}(\cos\theta)\sin^2\theta{\rm d}\theta
=\int_{0}^{\pi}(1-\cos^2\theta)^{\frac{1}{2}}\frac{{\rm d }P_{n}(\cos\theta)}{{\rm d} \cos\theta}\sin^2\theta{\rm d}\theta\nonumber\\
&=\int_{0}^{\pi}\sin\theta\frac{{\rm d }P_{n}(\cos\theta)}{{\rm d} \cos\theta}(1-\cos^2\theta){\rm d}\theta
=\int_{-1}^{1}\frac{{\rm d }P_{n}(t)}{{\rm d} t}(1-t^2){\rm d}t\nonumber\\
&=(1-t^2)P_{n}(t)\big|_{-1}^{1}+2\int_{-1}^{1}tP_{n}(t){\rm d}t
=2\int_{-1}^{1}P_{1}(t)P_{n}(t){\rm d}t
=
\left\{\begin{array}{lll}
	\frac{4}{2+1}, && n=1,\vspace{3pt}\\
	0,&& n\neq1.
\end{array}\right.
\end{align}
Hence, from  \eqref{NQ-4} and \eqref{NQ-5}, we can deduce 
\begin{align}\label{NQ-6}
\alpha_{n}^{\pm1}
&=\left\{\begin{array}{lll}
	-N_n^{\pm1} \kappa\bigg(\frac{N^{m_1}}{N^{m_1}}
-
\frac{N^{m_2}}{N^{m_2}}
\bigg)\times 
\begin{pmatrix}
\big(\int_{0}^{\pi}P_{1}^{\pm1}(\cos\theta)\sin^2\theta{\rm d}\theta\big)\big(\int_{0}^{2\pi} \cos^2\phi  {\rm d}\phi\big)\\
\mp{\rm i}\big(\int_{0}^{\pi}P_{1}^{\pm1}(\cos\theta)\sin^2\theta{\rm d}\theta\big)\big(\int_{0}^{2\pi} \sin^2\phi {\rm d}\phi\big)\\
 0
\end{pmatrix}, && n=1,\vspace{3pt}\\
	-N_n^{\pm1} \kappa\bigg(\frac{m_1 \int_{0}^{R}j_n(\kappa r)j_{0}^{m_1-1}(\kappa r)j_1(\kappa r)r^2{\rm d}r}{N^{m_1}}
-
\frac{m_2 \int_{0}^{R}j_n(\kappa r)j_{0}^{m_2-1}(\kappa r)j_1(\kappa r)r^2{\rm d}r}{N^{m_2}}
\bigg)\times \begin{pmatrix}
0\\
0\\
0
\end{pmatrix},&& n\neq1,
\end{array}\right.\nonumber\\
&=\left\{\begin{array}{lll}
	0\times \begin{pmatrix}
\frac{4}{3}\big(\int_{0}^{2\pi} \cos^2\phi  {\rm d}\phi\big)\\
\mp{\rm i}\frac{4}{3}\big(\int_{0}^{2\pi} \sin^2\phi {\rm d}\phi\big)\\
 0
\end{pmatrix}, && n=1,\vspace{3pt}\\
\begin{pmatrix}
0\\
0\\
0
\end{pmatrix},&& n\neq1,
\end{array}\right.
=0.
\end{align}

It follows from \eqref{NQ-0},  \eqref{NQ-3},  and \eqref{NQ-6}
that $\alpha_n^{m}=0$ for all $n = 0,1,\dots, m = -n, \dots, n$. From \eqref{EINE-11} and \eqref{JgQ}, we obtain
\begin{align*}
\zeta_n^{m}
&=-\kappa^{-2}\int_{B_R}\big(\nabla (j_{n}(\kappa|y|)\overline{Y_{n}^{m}({\hat y})})\big)\nabla\cdot{\boldsymbol  J}(y){\rm d}y\nonumber\\
&=\kappa^{-2}\int_{B_R}-\kappa^2 (j_{n}(\kappa|y|)\overline{Y_{n}^{m}({\hat y})}){\boldsymbol  J}(y){\rm d}y
=-{\alpha}_n^{m}=0.
\end{align*}

\begin{remark}\label{rem-J}
For $\kappa>0$, $R>0$, and $s>2$, consider $j_0(\kappa R)=0$ and
\begin{align*}
\boldsymbol  J(x)&=\left\{
\begin{aligned}
&\nabla j_0^{s}(\kappa |x|), &\quad |x|<R,\\
&0, &\quad |x|\geq R.
\end{aligned}
\right.
\end{align*}
Similarly, it can be shown that the function $\boldsymbol  J\in \mathcal{N}_2(R)$ is a nonradiating source, where 
$\alpha_{1}^{m}\neq 0, m=-1,0,1$ and $\alpha_n^{m}+\zeta_n^{m}=0$ for all $n = 0,1,\dots, m = -n, \dots, n$.
\end{remark}

\section{Conclusion}\label{S:co}

This paper explores the existence and characteristics of nonradiating sources in the context of Maxwell's equations. A nonradiating source is defined as a source that induces an electromagnetic field to be zero outside a finite region. This concept is closely associated with nonuniqueness in the inverse source problem. Starting from different governing equations for the electric field and electric current density, we develop various characterizations to elaborate the properties of nonradiating sources, taking into account their varying degrees of regularity. Moreover, the characterizations are investigated with respect to far-field patterns and near-field data of the electric field, as well as the null spaces of integral operators. The study involves the explicit construction of several illustrative examples demonstrating the existence of nonradiating sources in different null spaces. This work establishes a framework for comprehending the design or identification of specific arrangements of electric current densities, ensuring that they do not emit electromagnetic radiation.

\end{document}